\documentclass[review]{elsarticle}

\usepackage{hyperref}
\usepackage{latexsym,enumerate}
\usepackage{amsmath,amsthm,amsopn,amstext,amscd,amsfonts,amssymb}
\usepackage{xcolor}
\usepackage{verbatim} 
\usepackage{amsmath,tikz}
\usetikzlibrary{matrix}
\usepackage{amsmath,graphicx}

\newtheorem{theorem}{Theorem}[section]
\newtheorem{lemma}[theorem]{Lemma}

\newtheorem{proposition}[theorem]{Proposition}
\newtheorem{corollary}[theorem]{Corollary}

\newtheorem{remark}[theorem]{Remark}

\renewcommand{\d}{\delta}
\newcommand{\g}{\gamma}
\newcommand{\D}{\Delta}

\newcommand{\p}{\partial}

\journal{Journal of TBA}

\begin{document}

\begin{frontmatter}

\title{On the Boundary Polynomial of a Graph}


\author[FIU]{Walter Carballosa}
\ead{wcarball@fiu.edu}

\author[FIU]{Marcos Masip}
\ead{mmasi013@fiu.edu}

\author[BC]{Francisco A. Reyes}
\ead{freyes@broward.edu}

\address[FIU]{Department of Mathematics and Statistics, Florida International University,
Miami, Florida 33199, USA}

\address[BC]{Mathematics Department, Central Campus, Broward College, Davie, Florida 33314, USA}


\begin{abstract}
In this work, we introduce the boundary polynomial of a graph $G$ as the ordinary generating function in two variables $B(G;x,y):= \displaystyle\sum_{S\subseteq V(G)} x^{|B(S)|}y^{|S|}$, where $B(S)$ denotes the outer boundary of $S$.
We investigate this graph polynomial obtaining some algebraic properties of the polynomial. We found that some parameters of $G$ are algebraically encoded in $B(G;x,y)$, \emph{e.g.}, domination number, Roman domination number, vertex connectivity, and differential of the graph $G$. Furthermore, we compute the boundary polynomial for some classic families of graphs. We also establish some relationships between $B(G;x,y)$ and $B(G^\prime;x,y)$ for the graphs $G^\prime$ obtained by removing, adding, and subdividing an edge from $G$. In addition, we prove that a graph $G$ has an isolated vertex if and only if its boundary polynomial has a factor ($y+1$). Finally, we show that the classes of complete, complete without one edge, empty, path, cycle, wheel, star, double-star graphs, and many others are characterized by the boundary polynomial.
\end{abstract}

\begin{keyword}
\texttt{graph polynomial; generating function; boundary of a vertex set; vertex connectivity; vertex domination.}
\MSC[2017] 05C31\sep 05C69
\end{keyword}

\end{frontmatter}


\section{Introduction.}
Polynomials related to a graph have been a powerful tool to study structural, topological, and combinatorial properties of the graph. 
Graph polynomials have been widely studied since, in 1912, George D. Birkhoff introduced the chromatic polynomial in an attempt to prove the four-color theorem \cite{Bi}. 
Later in 1932 the chromatic polynomial was generalized to the Tutte--Whitney polynomial (or Tutte polynomial, or dichromatic polynomial), a polynomial that plays an important role in graph theory.

 In particular, graph polynomials are interesting when they encode much or essential information about the underlying graph. Several of the well-known polynomials are defined or can be written as an ordinary generating function of graph parameter(s), for example, the polynomials associated with chromatic numbers \cite{R,T}, defensive alliances \cite{CHRT,CRST}, differential of vertex sets \cite{BCLS}, domination sets \cite{AB}, independent sets \cite{BDN,GH}, induced subgraphs \cite{TAM}, matching sets \cite{F,GG}, and many others.

 The \emph{outer boundary} of a subset $S$ of vertices in a graph $G$ is the set of vertices in $G$ that are adjacent to vertices in $S$, but not in $S$ themselves. The \emph{inner boundary} is the set of vertices in $S$ that have a neighbor outside $S$. The \emph{edge boundary} is the set of edges with one endpoint in the inner boundary and one endpoint in the outer boundary. In this work, we choose the outer boundary of the vertex sets of a graph and its cardinality to define the boundary polynomial of a graph as an ordinary generating function in two variables.

We begin by stating the terminology used.
Throughout this paper, $G=(V,E)$ will denote a simple graph with order $n:=|V|$ and size $m:=|E|$. 
We write $u\sim v$ whenever $u$ and $v$ are adjacent vertices in $G$.
The edge joining these vertices is denoted by $uv$.
For a nonempty set $X\subseteq V$, and a vertex $v\in V$, $N_X(v)$ denotes
the set of neighbors that $v$ has in $X$, \emph{i.e.} $N_X(v):=\{u\in X: u\sim v\}$;
the degree of $v$ in $ X$  is denoted by $d_{X}(v):=|N_{X}(v)|$.
For a vertex $v\in V$, $N(v)$ denotes the set of all neighbors that $v$ has
in $V$, \emph{i.e.}, $N(v):=\{u\in V \; | \; u\sim v\}$; and $N[v]$ denotes
the closed neighborhood of the vertex $v$, \emph{i.e.}, $N[v]:= N(v)\cup \{v\}$.
We denote by $d_G(v_i):=|N(v_i)|$ the degree of a
vertex $v_i\in V$ in $G$, and by $\d(G), \D(G)$ the minimum and maximum degree of a vertex in $G$, respectively; when $G$ is clearly determined in the context, $G$ may be omitted.

The complement of $S\subseteq V$ is
denoted by $\overline{S}:=V\setminus S$ and the complement of $G$ is denoted by $\overline{G}$.
The subgraph induced by $S$ is denoted by $\langle S\rangle$. 
For $X,Y\subseteq V$, $E(X,Y)$ denotes the set of edges joining a vertex in $X$ with a vertex in $Y$.
For a vertex $v\in V$, the graph obtained from $G$ by removing the vertex $v$ and all edges joined to $v$ is denoted by $G-v$, and similarly, $G-e$ denotes the graph obtained from $G$ by removing the edge $e\in E$. 
A set of vertices $S\subseteq V$ is said to be \emph{dominating} if every vertex $v\in V\setminus S$ satisfies $N_{S}(v)\neq\emptyset$. The \emph{domination number}, $\g(G)$, is the minimum cardinality of a dominating set in $G$. 
A graph is said to be \emph{disconnected} if there are two vertices $u,v\in V$ with no path joining the vertices $u$ and $v$; otherwise, $G$ is \emph{connected}. If $G$ is disconnected, a \emph{connected component of} $G$ is a maximal (by inclusion) connected induced subgraph of $G$. 
A \emph{cut vertex} of a graph $G$ is a vertex $v$ such that the graph $G-v$ has more connected components than $G$. In particular, a cut vertex of a connected graph is a vertex whose deletion results in a disconnected graph.
A \emph{vertex cut} or \emph{separating set} of a graph $G$ is a set of vertices whose removal increases the number of connected components or reduces a connected component to $E_1$. If $G$ is connected, removing a vertex cut makes $G$ disconnected or reduced to $E_1$. The \emph{vertex connectivity} of $G$, denoted by $k_v(G)$, is the size of a smallest vertex cut. A graph is called $k$-\emph{vertex-connected} or $k$-\emph{connected} if its vertex connectivity is $k$ or greater.

As usual, ${\bf 1}_X$ denotes the \emph{indicator function} of $X$, \emph{i.e.}, ${\bf 1}_X(x)=1$ if $x\in X$ and ${\bf 1}_X(x)=0$ otherwise. 
We denote the polynomial coefficient of the term $x^k$ in a polynomial of two variables $P(x,y)$ by $[x^k]P(x,y)$, the polynomial coefficient of the term $y^r$ by $[y^r]P(x,y)$ and the coefficient of the term $x^ky^r$ by $[x^ky^r]P(x,y)$.

In the next section, we introduce the boundary polynomial, obtaining some algebraic properties of this polynomial of general graphs. We will study its coefficients and show some numerical evaluations related to graph parameters.
In Section 3, we investigate the distortion of the boundary polynomial under certain transformations such as the join of two graphs, the corona product of graphs, vertex addition, removal of an edge, and edge subdivision. In addition, we will compute the boundary polynomial of some classic graphs, \emph{e.g.}, complete, complete without an edge, complete bipartite, empty, path, cycle, wheel, star, and double-star graphs. 
Furthermore, we prove that complete, complete without an edge, empty, path, cycle, wheel, star, and double star graphs are characterized by their boundary polynomials.

\section{The Boundary Polynomial of a Graph}\label{sect:Poly}
Let $G$ be a graph of order $n$. We define the \emph{boundary polynomial} of $G$ with variable $x$ and $y$ as follows:
\begin{equation}\label{eq:Poly1}
    B(G;x,y) = \sum_{S \subseteq V(G)}x^{|B(S)|}y^{|S|}.
\end{equation}

Another way to define $B(G;x,y)$ is

\begin{equation}\label{eq:Poly2}
    B(G;x,y) = \sum_{0 \leq i+j \leq n}B_{i,j}(G)x^{i}y^{j},
\end{equation}

where $B_{i,j}(G)$ represents the number of subsets of the vertex in $G$ with $|B(S)| = i$ and $|S| = j$. As long as there is no confusion, we will omit the reference to $G$ in the coefficients $B_{i,j}$. 
Indeed, the boundary polynomial is directly related to the differential polynomial in one variable, defined in \cite{BCLS}, when we make $y:=x^{-1}$, \emph{i.e.}, 
\begin{equation}\label{eq:1var_2var}
  B(G;x)=x^n\,B(G;x,x^{-1}).
\end{equation}
We recall that a {\em Roman dominating function} on $G$ is a function $f:V(G)\rightarrow \{0,1,2\}$ satisfying the condition that every vertex $u$ for which $f(u)=0$ is adjacent to at least one vertex $v$ for which $f(v)=2$. The \emph{weight} of a Roman dominating function $f$ is the value $w(G,f)=\sum_{u\in V(G)}f(u)$. The minimum weight of a Roman dominating function in $G$ is the {\em Roman domination number} of $G$ and is denoted by $\g_R(G)$. 
We also recall that \emph{differential of a graph} $G$, usually denoted by $\partial(G)$, is the maximum of $|B(S)|-|S|$ over all vertex sets $S\subseteq V$. In \cite{BFS}, Bermudo \emph{et al.} proved that for every graph $G$ we have $\g_R(G)+\partial(G)=n$. 
Thus, Basilio \emph{et al.}, in \cite{BCLS}, showed that the degree of the differential polynomial in one variable relates the Roman domination number and the differential of the graph, \emph{i.e.}, $deg\big(B(G;x)\big)=\gamma_R(G)+2\partial(G)$. Therefore, we have
\begin{equation}\label{eq:Poly_deg}
    \partial(G) = deg\big(B(G;x,x^{-1})\big) \qquad \text{ and } \qquad \gamma_R(G) = deg\big(B(G;x,x)\big)- deg\big(B(G;x,x^{-1})\big).
\end{equation}

\medskip
%
%
%
%
%
%
%

Note that given the adjacency matrix of a graph $G$, we can compute the cardinality and outer boundary of each subset of vertices and obtain the boundary polynomial in computational order $O(n^22^n)$. 
For example, we compute the boundary polynomial $B(K_{4};x,y)$ of the complete graph $K_{4}$: When $S = \emptyset$, we have $|B(S)| = |S| = 0$. Now, considering $\emptyset \not= S \subseteq V(K_{4})$ we note that for $i \in \{1,2,3,4\}$ there are $\binom{4}{i}$ subsets of $V(K_4)$ with cardinality $i$, it follows easily that  $B(K_{4};x,y) = 4x^{3}y + 6x^{2}y^{2} + 4xy^{3} + y^{4} + 1$. Notice that $x^{4}B(K_{4};x,x^{-1}) = B(K_{4};x) = 4x^{6}+7x^{4}+4x^{2}+1$, which is the differential polynomial of $K_{4}$ of one variable \cite{BCLS}.

We recall that two graphs $G_1$ and $G_2$ are \emph{isomorphic}, denoted $G_1 \simeq G_2$, if there exists a bijection $f: V(G_1)\to V(G_2)$ such that any two vertices $u$ and $v$ of $G_1$ are adjacent if and only if $f(u)$ and $f(v)$ are adjacent in $G_2$. Note that if $G_1$ and $G_2$ are isomorphic graphs, then $B(G_1; x,y) = B(G_2; x,y)$. 
Similarly in \cite[Theorem 2.9]{BCLS}, we have the following result.

\begin{theorem}\label{t:factors}
  Let $G$ be a non-connected graph with $k>1$ connected components, ${G_1,G_2,\ldots,G_k}$, then we have
\begin{equation}\label{eq:factors}
  B(G;x,y)=\displaystyle\prod_{i=1}^{k} \,B(G_i;x,y).
\end{equation}
\end{theorem}

\begin{proof}
Consider the bijection $\phi:\mathcal{P}\big(V(G)\big)\to{V(G_1)\times V(G_2)\times\ldots\times V(G_k)}$, defined for every $S\subseteq V(G)$ as follows that
\[
 \phi(S)=(S_1,S_2,\ldots,S_k) \quad \text{ where } \quad S_i:=S\cap V(G_i) \quad\text{for} \quad 1\le i\le k.
\]
Then, using associativity, commutativity and distributivity properties of real numbers we have
\begin{align*}
    B(G;x,y)=\displaystyle\sum_{S\subseteq V(G)} x^{|B(S)|}\, y^{|S|} = \sum_{\phi(S)\subseteq \phi\big(V(G)\big)} \left( \prod_{i=1}^{k} x^{|B(S_i)|}\, y^{|S_i|} \right)= \prod_{i=1}^{k} \left(\sum_{S\subseteq V(G_i)} x^{|B(S)|}\, y^{|S|} \right)
\end{align*}
\end{proof}

The following theorem shows how the coefficients of $B(G;x,y)$ can provide useful information about $G$. In the remainder of the paper, we adopt the convention that $\binom{k}{r}=0$ when $k<r$.

\begin{theorem}\label{t:coef}
Let $G$ be a graph of order $n$. The coefficients of $B(G;x,y)$ satisfy the following properties:
 \begin{enumerate}[i)]
    \item $B_{0,j}(G) = B_{0,n-j}(G)$ for $0 \leq j \leq n$ and $B_{0,0}(G) = B_{0,n}(G) = 1$.
 
     \item $G$ is connected if and only if $B_{0,j}(G) = 0$ for $1 \leq j \leq n-1$. 
     
     \item $B_{0,1}(G) = B_{0,n-1}(G)$ is the number of isolated vertices of $G$.
     
     \item $B_{d,1}(G)$ is the number of vertices of $G$ with degree $d$. 
     
     \item $n=[y]B(G;x,y)\big|_{x=1}$. 
     
     \item $m=\displaystyle\frac12\,\left.\frac{\text{d}\big([y]B(G;x,y)\big)}{\text{d}x}\right|_{x=1}$

     \item The number of components of $G$ isomorphic to $P_{2}$ is $p:=B_{0,2}(G)-\binom{B_{0,1}(G)}{2}$.

     \item The number of components of $G$ isomorphic to $P_{3}$ or $C_{3}$ is $q:=B_{0,3}(G) - \binom{B_{0,1}(G)}{3} - p\cdot B_{0,1}(G)$.

     \item $\g(G)$ is the minimum $1\le k\leq n-1$ such that $[x^{n-k}y^{k}]B(G;x,y)\neq0$.

     \item If $G$ is connected, then $k_v(G)$ is the minimum $1\le k\le n-1$ such that $\displaystyle\left.\frac{\text{d}^{k} B(G;x,y)}{\text{d}x^k}\right|_{(x,y)=(0,1)}\neq0$.

     \item For every $k\leq\d$, we have $[x^ky^{n-k}]B(G;x,y)=\binom{n}{k}$; furthermore, $[x^ry^{n-k}]B(G;x,y)=0$ if $r\neq k$.
 \end{enumerate}
\end{theorem}

\begin{proof}
We prove separately each item.

\begin{enumerate}[$i)$]
  \item {Let $S$ be a subset of vertices in $G$ such that $|S| = j$ and $|B(S)| = 0$. Hence, $S$ has no adjacent vertex in $\overline{S}$; therefore, $\overline{S}$ has no adjacent vertex in $S$ as well. Indeed, for every $S\subseteq V(G)$ counted in $B_{0,j}(G)$, $\overline{S}$ is counted in $B_{0,n-j}(G)$, and vice versa. Thus, $B_{0,j}(G) = B_{0,n-j}(G)$ for $1 \leq j \leq n-1$. On the other hand, $B_{0,0}(G) = 1$ since it only counts $S=\emptyset$, and $B_{0,n}(G)=1$ only counts $S=V(G)$.
  }

  \item {If $G$ is connected, then such a vertex set $S$ as in part i) does not exist, so $B_{0,j}(G) = 0$ for $1 \leq j \leq n-1$. If $G$ is not connected, then there is $\emptyset\neq S\subseteq V(G)$ having no adjacent vertex in $\overline{S}$; therefore, $B_{0,j}(G) > 0$ for some $1 \leq j \leq n-1$.}

  \item {Clearly, $B_{0,1}(G)$ counts all vertices with no neighbors, so it is exactly the number of isolated vertices in $G$. The equality follows from Theorem \ref{t:coef} part i).}

  \item {Similarly to part iii) above, $B_{d,1}(G)$ counts all vertices with $d$ neighbors in $G$.}

  \item {By Theorem \ref{t:coef} part iv), the sum of all coefficients of the polynomial in $x$ that is the coefficient of $y$ in $B(G;x,y)$ is the total number of vertices in $G$, therefore, $n$ equals the evaluation of the polynomial $[y]B(G;x,y)$ at $x=1$.}

  \item {By the Handshake Lemma the size of $G$ is half of the sum of the degree of the vertices of $G$. Thus, it is 
  \[
    n=\displaystyle\frac12\left(\sum_{1}^{n-1} B_{i,1}(G)\cdot i\right)=\frac12\left(\sum_{1}^{n-1} B_{i,1}(G)\cdot i\cdot1^{i-1}\right)=\frac12\,\left.\frac{\text{d}\big([y]B(G;x,y)\big)}{\text{d}x}\right|_{x=1}
  \]
  }

  \item {By definition $B_{0,2}(G)$ counts all $S\subseteq V(G)$ with $|S|=2$ and $E(S,\overline{S})=\emptyset$. Then, $B_{0,2}(G)$ counts all connected components of $G$ isomorphic to $P_2$ as well as all couples of isolated vertices of $G$.}

  \item {By definition $B_{0,3}(G)$ counts all $S\subseteq V(G)$ with $|S|=3$ and $E(S,\overline{S})=\emptyset$. Then, $B_{0,3}(G)$ counts all the connected components of $G$ isomorphic to $P_{3}$ or $C_{3}$ as well as all sets of three isolated vertices of $G$ and pairs of isolated vertices with a connected component isomorphic to $P_2$.}

  \item This result follows from the fact that $S\subseteq V$ is dominating if and only if $B(S)=\overline{S}$. So, $\g(G)$ is the smallest $k$ such that the term $x^{n-k}y^k$ appears in $B(G;x,y)$.

  \item {Since $G$ is connected $B_{1,n-1}(G)=n\neq0$, thus, such a minimum for $k$ exists. Let $x^k\,y^r$ be the variable part of a monomial with non-zero coefficient in $B(G;x,y)$ such that $k+r< n$. Hence, there is $S\subseteq V(G)$ such that $|S|=r$, $|B(S)|=k$ and $\hat{S}:=V(G)\setminus\big(S\cup B(S)\big)\neq\emptyset$. Thus, $B(S)$ is a vertex cut. Conversely, if $B(S)$ is a vertex cut with $|S|=r$ and $|B(S)|=k$, then $B_{k,r}(G)\neq0$ and $k+r<n$. Therefore, if $k+r<n$, the minimum $k$ holds when $|B(S)|$ is minimum, \emph{i.e.}, $r=k_v(G)$. If $B_{k,r}(G)=0$ for every $1\le r,k\le n-1$ such that $k+r<n$, then $G$ has no vertex cut, $G\simeq K_n$ and consequently $k_v(G)=n-1$ which is the minimum $k$.}

  \item {Consider $S\subseteq V(G)$ with $|S|=n-k\geq n-\d$, then we have $\left|\overline{S}\right|\leq\d$. Therefore, $B(S)=\overline{S}$ since every vertex $v\in\overline{S}$ has a neighbor in $S$ since $k\leq\d$.}
\end{enumerate}
\end{proof}

The following results are well-known, which can also be obtained from the algebraic properties of the boundary polynomial of disconnected graphs. They are a direct consequence of Theorems \ref{t:factors} and \ref{t:coef} and \cite[Theorem 1]{BFS}.

\begin{corollary}\label{c:factors}
   Let $G$ be a non-connected graph with $k>1$ connected components, ${G_1,G_2,\ldots,G_k}$, then we have
  \begin{equation}\label{eq:c_factors}
   \begin{aligned}
       \p(G)&=\displaystyle\sum_{i=1}^{k} \p(G_i) \\
       \g_R(G)&=\displaystyle\sum_{i=1}^{k} \g_R(G_i) \\
       \g(G)&=\displaystyle\sum_{i=1}^{k} \g(G_i) \\
       k_v(G)&=\displaystyle\min_{1\le i\le k} k_v(G_i)
   \end{aligned}
  \end{equation}
\end{corollary}

The following result shows that the boundary polynomial encodes basic information about the graph as seen through simple evaluations.

\begin{theorem}\label{t:eval}
 Let $G$ be a graph with order $n$ and size $m$. Then, $B(G;x,y)$ satisfies the following properties:

   \begin{enumerate}[i)]
     \item $B(G;1,y)=(1+y)^n$, in particular, we have $B(G;1,-1)= 0$ and $B(G;1,1)=2^n$.

     \item If $G$ has $k\ge1$ connected components with respective orders $n_1,\ldots,n_k$. Then the number of connected components of $G$ is $k=\log_2 B(G;0,1)$ and the orders of the components of $G$ are determined by its factorization $B(G;0,y)=\displaystyle\prod_{i=1}^{k}\,(1+y^{n_i})$.

     \item $m=\displaystyle\frac{\;\displaystyle\left.\frac{\text{d}[y^2]B(G;x,y)}{\text{d}x}\right|_{x=1}-\frac12\,\left.\frac{\text{d}^2\big([y]B(G;x,y)\big)}{\text{d}x^2}\right|_{x=1}\;}{2(n-2)}$

   \end{enumerate}
\end{theorem}

\begin{proof}
We prove separately each item.
\begin{enumerate}[$i)$]
  \item {By \eqref{eq:Poly2} we have 
  $$B(G;1,y) = \sum_{0 \leq i+j \leq n}B_{i,j}(G)1^{i}y^{j}=\sum_{j=0}^{n}\,\left(\sum_{i=0}^{n}B_{i,j}(G)\right)\, y^{j}.$$ 
  However, we have that $\sum_{i=0}^{n}B_{i,j}(G)=\binom{n}{j}$ since the inner sum above counts all subsets of vertices with exactly $j$ vertices. Therefore, we obtain $B(G;1,y) = \displaystyle\sum_{j=0}^{n}\,\binom{n}{j}\, y^{j}=(1+y)^n$. The particular cases follow from direct evaluation.}

  \medskip

  \item {By Theorems \ref{t:factors} and \ref{t:coef} ii) it is easy to obtain $B(G;0,y)=\displaystyle\prod_{i=1}^{k}\,(1+y^{n_i})$. Therefore, we have $B(G;0,1)=\displaystyle\prod_{i=1}^{k}\,(1+1^{n_i})=2^k$.}

  \medskip

  \item {Consider $S\subseteq V(G)$ with $S=\{u,v\}$ and $u\neq v$. Hence, we have 
  $$|B(S)|=d(u)+d(v)-|N(u)\cap N(v)|-2\cdot{\bf 1}_{u\sim v}$$
  Moreover, we have $[y^2]B(G;x,y)(x)=\displaystyle\sum_{S\subseteq V(G), |S|=2} x^{|B(S)|}$, and 
   \begin{align*}
    \displaystyle\left.\frac{\text{d}[y^2]B(G;x,y)}{\text{d}x}\right|_{x=1} & =\sum_{u,v\in V(G), u\neq v}\, \Big(d(u)+d(v)-|N(u)\cap N(v)|-2\cdot{\bf 1}_{u\sim v}\Big) \\
    & = \sum_{v\in V(G)}\left( (n-1)d(v) - \binom{d(v)}{2} \right)  - 2m \\
    & = 2m(n-2)- \sum_{v\in V(G)} \binom{d(v)}{2}.
   \end{align*}
  Finally, since $[y]B(G;x,y)=\displaystyle\sum_{v\in V(G)} x^{d(v)}$ we have $\left.\frac{\text{d}^2\big([y]B(G;x,y)\big)}{\text{d}x^2}\right|_{x=1}=\displaystyle\sum_{v\in V(G)} d(v)\big(d(v)-1\big)$ which completes the proof. 
  
  }
%
%
\end{enumerate}
\end{proof}

Here are two more interesting polynomials derived from the boundary polynomial 
\[
 B(G;x,1):=\displaystyle\sum_{S\subseteq V(G)} x^{|B(S)|}\qquad \text{ and } \qquad B(G;x,-1):=\displaystyle\sum_{S\subseteq V(G)} x^{|B(S)|}\cdot (-1)^{|S|}.
\]
We will study them.
The following result characterizes the graphs with isolated vertices by the occurrence of the factor $(y+1)$ in $B(G;x,y)$.

\begin{theorem}\label{t:isolate}
  $G$ has an isolated vertex if and only if $B(G;x,y)$ has a factor $(y+1)$. 
  
  Furthermore, The multiplicity of the factor $(y+1)$ in $B(G;x,y)$ is the number of isolated vertices in $G$.
\end{theorem}

\begin{proof}
    On one hand, if $G$ has an isolated vertex, Theorem \ref{t:factors} gives that $B(G;x,y)$ has a factor $(y+1)$ since $G$ has a connected component isomorphic to the complete graph $K_1$.

    \medskip

    On the other hand, if $B(G;x,y)$ has a factor $(y+1)$, then we have $B(G;x,-1)=0$ for every $x\in\mathbb{R}$. Thus, each coefficient of $B(G;x,-1)$ is zero, in particular, the coefficient of $x$, \emph{i.e.}, $$[x]B(G;x,-1)=\displaystyle\sum_{S\subseteq V(G),|B(S)|=1} (-1)^{|S|}=0.$$ 
    Now we claim that \emph{if $[x]B(G;x,-1)=0$ then $G$ has an isolated vertex.} We prove the contrapositive, \emph{i.e.}, showing that \emph{if $G$ has no isolated vertices then $[x]B(G;x,-1)\neq0$.}

    So, assume that $G$ has no isolated vertices. Hence, every $S\subseteq V(G)$ with $|B(S)|=1$ satisfies 
    \begin{enumerate}[i)]
        \item $|S|=n-1$, or
        \item $v$ is a cut vertex where $B(S)=\{v\}$.
    \end{enumerate}
    Note that if $|S|=n-1$, then $|B(S)|=1$ since $G$ has no isolated vertex. Besides, if $B(S)=\{v\}$ and $|S|<n-1$, then $V(G)\setminus (S\cup\{v\})\neq\emptyset$; therefore, $v$ is a cut vertex of $G$ since $E\Big(S,V(G)\setminus (S\cup\{v\})\Big)=\emptyset$.  
    Denoted by $[x]B^{(v)}(G;x,-1)$, the corresponding sum within $[x]B(G;x,-1)$ associated to $B(S)=\{v\}$. Note that $$[x]B(G;x,-1)=\displaystyle\sum_{v\in V(G)} [x]B^{(v)}(G;x,-1).$$ 
    So, if $v$ is not a cut vertex, then $[x]B^{(v)}(G;x,-1)=(-1)^{n-1}$.
    Furthermore, if $G$ has no cut vertex, then 
    \begin{align}\label{eq:sum_v}
     [x]B(G;x,-1)=\displaystyle\sum_{v\in V(G)} (-1)^{n-1}=(-1)^{n-1}\,n\neq0.
    \end{align}
    Now consider if $G$ has a cut vertex. Consider $v$ a cut vertex of $G$ such that $G-\{v\}$ has $k>1$ connected components $G_1,G_2,\ldots,G_k$ with respective orders $n_1,n_2,\ldots,n_k$. Then, the number of $S\subseteq V(G)$ such that $B(S)=\{v\}$ is $2^k-1$ since those $S$ must be union of some of the vertex sets $V(G_1), V(G_2),\ldots, V(G_k)$. 
    
    If $n-1$ is odd, then $[x]B^{(v)}(G;x,-1)=(-1)^{n-1}$ since for every $S$ such that $B(S)=\{v\}$ we have that $S^\prime:=\overline{S}\setminus\{v\}$ satisfies that $B(S^\prime)=\{v\}$ except for $S=V(G)\setminus\{v\}$. Thus, we have that \eqref{eq:sum_v} also holds when $n-1$ is odd.
    
    Assume now that $n-1$ is even.  If $n_i$ is even for every $1\le i\le k$, then 
    $$[x]B^{(v)}(G;x,-1)=(2^k-1)(-1)^{n-1}=2^k-1\neq0.$$
    However, if out of $n_1,n_2,\ldots,n_k$ there are an even number of odd orders, let us consider $2r$ of them are odd and $k-2r$ are even. Without loss of generality we can assume that $n_1,n_2,\ldots,n_{2r}$ are odd and the others $k':=k-2r$ are even. If $k'=0$, the we have
    \[
      [x]B^{(v)}(G;x,-1)=\displaystyle\sum_{i=1}^{2r} \binom{2r}{i}\,(-1)^{i}=(1-1)^{2r}-1=-1.
    \]
    Otherwise, we can separately count those $S$ for which neither of $V(G_i)$ for $1\le i\le 2r$ are included in $S$ and when some of them are included, then we have
    \[
     [x]B^{(v)}(G;x,-1)= (2^{k'}-1) \,+\,2^{k'}\displaystyle\sum_{i=1}^{2r} \binom{2r}{i}\,(-1)^{i}=(2^{k'}-1) \,+ \,\cdot2^{k'}[(1-1)^{2r}-1]=-1.
    \]
    In fact, as a summary, for $n-1$ even, we have 
    \[
     [x]B^{(v)}(G;x,-1)=\left\{\begin{aligned}
         1 \qquad & \text{if } v \text{ is not a cut vertex}\\
         2^{k}-1 \qquad & \text{if } v \text{ is a cut vertex and all } k \text{ components are even} \\
         -1 \qquad & \text{if } v \text{ is a cut vertex and some components are odd}
     \end{aligned}\right.
    \]
    Thus, since $n$ is odd and $[x]B^v(G;x,-1)$ is also odd for every $v$, then we obtain that $[x]B(G;x,y)\neq0$.
    Therefore, proof for the if and only if part of the theorem is completed.

    Finally, by Theorem \ref{t:factors}, for each factor $(y+1)$ of $B(G;x,y)$, one at the time, we can remove an isolated vertex $u$ of $G$, and then we can apply to $G-{u}$ the previous result of this theorem. Thus, iterating this process as many times as isolated vertices of $G$ appear.
\end{proof}


\section{The Boundary Polynomial of Some Classes of Graphs}\label{sect3}

This section focuses on determining the boundary polynomial of some well-known families of graphs. In particular, we present explicit formulas for the boundary polynomial of complete, empty, path, cycle, wheel, and double star graphs, thereby expanding the number of classical graph classes with known boundary polynomials. 
Furthermore, we investigate the boundary polynomial of some graphs derived from others; \emph{i.e.}, the combination of two or more graphs as well as a few unary operations on a graph. 
As usual, we shall use $K_n, E_n,  P_n, C_n, W_n,$ and $S_n$ to denote, respectively, the complete, empty, path, cycle, wheel, and star graph of order $n$.  
In particular, note that $\overline{K}_n$ is isomorphic to the empty graph $E_n$.

\medskip

Theorem \ref{t:factors} allows us to obtain the boundary polynomial of $G\uplus \{v\}$ obtained by adding to $G$ a new isolated vertex $v$ \big(\emph{i.e.}, $v\notin V(G)$\big). This operation is called \emph{vertex addition}.

\begin{corollary}\label{c:G+v}
Let $v$ be a vertex such that $v\notin V(G)$. Then
$ B(G\uplus \{v\};x,y)= B(G;x,y)\; (y+1).$
\end{corollary}

Our next result is the consequence of Corollary \ref{c:G+v} and direct calculation.

\begin{proposition}\label{p:En+Kn}
   For every $n\ge1$ we have
   \begin{equation}\label{eq:En}
       B(E_n;x,y)=(1+y)^n
   \end{equation}
   \begin{equation}\label{eq:Kn}
       B(K_{n};x,y) = (x+y)^{n} + 1 - x^{n}
   \end{equation}
\end{proposition}

\begin{proof}
Note that \eqref{eq:En} is a direct consequence of Corollary \ref{c:G+v}. Since we have 
\[
 B(E_{n+1};x,y)= (y+1)\cdot B(E_n;x,y) \quad \text{ and }\quad B(E_1;x,y)= y+1.
\]

\smallskip

For \eqref{eq:Kn}, consider $S\subseteq V(K_n)$. 
Then, for every $1\le j\le n$ we have that if $|S|=j$, then $|B(S)|=|\overline{S}|=n-j$. Thus, we have
\[
  B(K_n;x,y)=B_{0,0}(K_n)+\displaystyle\sum_{j=1}^{n} B_{n-j,j}(K_n)\,x^{n-j}\,y^{j}=1+\displaystyle\sum_{j=1}^{n} \binom{n}{j} \, x^{n-j}\, y^{j} = (x+y)^n +1-x^n.
\]
\end{proof}

Let $H, H_1$ and $H_2$ be graphs, and let $*$ be a standard graph operation such that $H=H_1*H_2$.  A classical approach used when investigating some graph invariants is to establish relationships between $H_1,H_2$ and the resulting graph $H$ with respect to such an invariant. Now we shall use this strategy in order to obtain the boundary polynomial of several graphs. 
Let $G_1$ and $G_2$ be two disjoint graphs. We recall that the \emph{graph join} $G_1 + G_2$ of $G_1$ and $G_2$ is the graph that results by  adding to $G_1\uplus G_2$ all the edges of the form $uv$ with $u\in V(G_1)$ and $v\in V(G_2)$.
Note that the join is a commutative operation where $V(G_1+G_2)=V(G_1)\cup V(G_2)$.

\begin{theorem}\label{t:join}
  Let $G_1,G_2$ be two graphs of order $n_1$ and $n_2$, respectively. Then
  \[
  B(G_{1}+G_{2};x,y) = B(K_{n_1+n_2};x,y) \,+\, x^{n_{2}} \big(B(G_{1};x,y)-B(K_{n_1};x,y)\big) \,+\, x^{n_{1}} \big(B(G_{2};x,y)-B(K_{n_2};x,y)\big).
  \]
\end{theorem}  

\begin{proof}
 Let $S \subseteq V(G_{1}+G_{2})$. 
 
 If $S=\emptyset$, then $B(S)=\emptyset$ and we get the addend $x^0\,y^0=1$. Assume now that $S\neq\emptyset$.

 If $S\subseteq V(G_1)$, then $B_{G_1+G_2}(S)=B_{G_1}(S)\cup V(G_2)$. It yields the expression 
 \[
  \sum_{\emptyset \not= S \subseteq V(G_{1})} x^{|B_{G_{1}}(S)|+n_2}y^{|S|} = x^{n_{2}}\sum_{\emptyset \not= S\subseteq V(G_{1})} x^{|B_{G_{1}}(S)|}y^{|S|} = x^{n_{2}}\big(B(G_{1};x,y)-1\big).
 \]

By symmetry, if $S\subseteq V(G_2)$, we obtain 
 \[ 
   x^{n_{1}}\big(B(G_{2};x,y)-1\big).
 \]

Finally, if $S_1:=S\cap V(G_1)\neq\emptyset$ and $S_2:=S\cap V(G_2)\neq\emptyset$, then $B_{G_1+G_2}(S)=V(G_1+G_2)\setminus S$. Thus, it follows  
 \begin{align*}
  \quad & \sum_{\emptyset \not= S \subseteq V(G_{1}+G_{2})} x^{|B_{G_{1}+G_{2}}(S)|}y^{|S|}\,-\sum_{\emptyset\neq S\subseteq V(G_1)}\, x^{|B_{G_{1}+G_{2}}(S)|}y^{|S|}\,-\, \sum_{\emptyset\neq S\subseteq V(G_2)}\,x^{|B_{G_{1}+G_{2}}(S)|}y^{|S|} \\
  = & \sum_{j=1}^{n_1+n_2}\binom{n_1+n_2}{j} x^{n_1+n_2-j} y^{j} \,-\, x^{n_2} \sum_{j=1}^{n_1} \binom{n_1}{j} x^{n_1-j} y^{j}  \,-\, x^{n_1} \sum_{j=1}^{n_2} \binom{n_2}{j} x^{n_2-j} y^{j} \\
  = & \big(B(K_{n_1+n_2};x,y)-1\big) \,-\, x^{n_{2}} \big(B(K_{n_1};x,y)-1\big) \,-\, x^{n_{1}} \big(B(K_{n_2};x,y)-1\big).
\end{align*}
\end{proof}

\medskip

\begin{corollary}\label{c:E+G}
 Let $G$ be a graph with order $n\ge1$. Then, we have 
 \[
 B(E_1+G;x,y) = B(K_{n+1};x,y) \,+\, x \Big(B(G;x,y)-B(K_{n};x,y)\Big).
 \]
\end{corollary}

Our next result is a direct consequence of Theorem \ref{t:join} and the fact that $W_{n}\simeq E_1 + C_{n-1}$, $K_{n,m}\simeq E_n + E_m$, $S_{n}\simeq K_1 + E_{n-1}$ and $K_{n}-e\simeq E_2 + K_{n-2}$.

\begin{corollary}\label{p:join_samples} 
 {$\quad$}
 \begin{enumerate}[(i)]
   \item For every $n\ge4$ 
       \begin{equation}\label{eq:Wn}
         B(W_{n};x,y) = x\,B(C_{n-1};x,y)-x + y(x+y)^{n-1} + 1.
       \end{equation}

   \item For every $n,m\ge1$ 
       \begin{equation}\label{eq:k_nm}
         B(K_{n,m};x,y) = (x+y)^{n+m} + x^{n}\big((1+y)^{m}-(x+y)^{m}\big) + x^{m}\big((1+y)^{n}-(x+y)^{n}\big) + \big(x^n-1\big)\big(x^m-1\big)
       \end{equation}

   \item For every $n\ge2$
       \begin{equation}\label{eq:Sn}
         B(S_n;x,y) = (x+y)^{n} + 1 + x\left((1+y)^{n-1}-(x+y)^{n-1}-1\right)
       \end{equation}

   \item For every $n\ge3$
       \begin{equation}\label{eq:kn-e}
         B(K_n-e;x,y) = B(K_{n};x,y) + 2x^{n-2}y(1-x).
       \end{equation}
  \end{enumerate}
\end{corollary}

These are also well-known results that can also be obtained from Theorems \ref{t:coef} and \ref{t:join}. 

\begin{corollary}\label{c:join_properties}
  Let $G_1,G_2$ be two graphs of order $n_1$ and $n_2$, respectively. Then
  \begin{equation}\label{eq:join_properties}
   \begin{aligned}
       \p(G_1+G_2)&=\displaystyle\left\{\begin{array}{cc} n_1+n_2-2 & \quad \text{ if } \p(G_i)=n_i-2 \text{ for some } i=1,2; \\ n_1+n_2-4 & \quad \text{ otherwise} \end{array}\right. \\
       \g_R(G_1+G_2)&=\displaystyle\left\{\begin{array}{cc} 2 & \qquad\qquad\qquad \text{ if } \g_R(G_i)\le2 \text{ for some } i=1,2; \\ 4 & \qquad\qquad\qquad \text{ otherwise} \end{array}\right. \\
       \g(G_1+G_2)&=\displaystyle\left\{\begin{array}{cc} 1 & \qquad\qquad\qquad \text{ if } \g(G_i)=1 \text{ for some } i=1,2; \\ 2 & \qquad\qquad\qquad \text{ otherwise} \end{array}\right. \\
       k_v(G_1+G_2)&=\displaystyle\min\{k_v(G_1)+n_2,k_v(G_2)+n_1\}
   \end{aligned}
  \end{equation}
\end{corollary}

For $v\in V(G)$ and $i\in\{0,1\}$ let us define the following
\
\begin{equation}\label{eq:def1}
\begin{aligned}
 B(G;x,y)^{i}_v & :=\displaystyle \sum_{S\subseteq V,
{\bf 1}_S(v)=i} \, x^{|B(S)|}y^{|S|},
\\
B(G;x,y)^{i}_{v^{0}}& :=\displaystyle \sum_{S\subseteq
V,{\bf 1}_S(v)=i, N(v)\cap S=\emptyset} \, x^{|B(S)|}y^{|S|},
\\
B(G;x,y)^{i}_{v^{1}} & :=\displaystyle \sum_{S\subseteq
V,{\bf 1}_S(v)=i,N(v)\cap S \neq\emptyset} \, x^{|B(S)|}y^{|S|}.
\end{aligned}
\end{equation}

\begin{remark}\label{remark:frank} 
  The following equalities follow from previous definitions,
\begin{eqnarray*}
B(G;x,y)& = & B(G;x,y)^{0}_{v}+B(G;x,y)^{1}_{v},\\
B(G;x,y)^{0}_{v}& = & B(G;x,y)^{0}_{v^{0}}+B(G;x,y)^{0}_{v^{1}}\\
B(G;x,y)^{1}_{v}& = & B(G;x,y)^{1}_{v^{0}}+B(G;x,y)^{1}_{v^{1}}.
\end{eqnarray*}
\end{remark}
\medskip

Let $G$ be a graph. Consider a vertex $v\in V(G)$ and a vertex $u\notin V(G)$. Let $G_e$ be the graph obtained by adding to $G$ the vertex $u$ and the edge $e=uv$ such that $v$ is a pendant vertex of $G_e$. Our next result provides a relation between $B(G;x,y)$ and $B(G_{e};x,y)$.

\begin{lemma}\label{l:pendent+}
 Let $e=uv$, $G$ and $G_{e}$ be as above. Then
 \[
 \begin{bmatrix}
     B(G_{e};x,y)^{0}_{u^{0}} \\ B(G_{e};x,y)^{0}_{u^{1}} \\ B(G_{e};x,y)^{1}_{u^{0}} \\ B(G_{e};x,y)^{1}_{u^{1}}
 \end{bmatrix}
 =
 \begin{bmatrix}
     1 & 1 & 0 & 0 \\ 0 & 0 & x & x \\ xy & y & 0 & 0 \\ 0 & 0 & y & y
 \end{bmatrix}
 \begin{bmatrix}
     B(G;x,y)^{0}_{v^{0}} \\ B(G;x,y)^{0}_{v^{1}} \\ B(G;x,y)^{1}_{v^{0}} \\ B(G;x,y)^{1}_{v^{1}}
 \end{bmatrix}.
 \]
\end{lemma}

\begin{proof}
Let $G_{e}$ be of order $n$, which implies that $G$ is of order $n-1$. We consider that $N_{G_{e}}(u) = \{v\}$, and that $N_{G}(v) = N_{G_{e}}(v)\setminus \{u\}$. Let $S^\prime := S\setminus \{u\} \subseteq V(G)$ for $S \subseteq V(G_{e})$. For $i \in \{0,1\}$ we define
$$\mathcal{V}_{u,v}^{i,j} := \{S \subseteq V(G_{e}): {\bf 1}_{S}(u) = i, {\bf 1}_{S}(v) = j\}.$$
We can see that $\left\{\mathcal{V}_{u,v}^{0,0}, \mathcal{V}_{u,v}^{0,1},\mathcal{V}_{u,v}^{1,0}, \mathcal{V}_{u,v}^{1,1}\right\}$ is a partition of the power set of $V(G_{e})$. Now, let $S \in \mathcal{V}_{u,v}^{0,0}$. Then $S = S^\prime$ and $B_{G_{e}}(S) = B_{G}(S^\prime)$. Because of the bijection between $\mathcal{V}_{u,v}^{0,0}$ and the collection $\{S^\prime \subseteq V(G): {\bf 1}_{S^\prime}(v) = 0\}$, then we have 
\begin{equation}\label{eq:BGE1}
    B(G_{e};x,y)_{u^{0}}^{0} = \sum_{S \in \mathcal{V}_{u,v}^{0,0}}x^{|B_{G_{e}}(S)|}\,y^{|S|} = \sum_{S^\prime \subseteq V(G),{\bf 1}_{S^\prime}(v) = 0}x^{|B_{G}(S^\prime)|}\,y^{|S^\prime|} = B(G;x,y)_{v}^{0}.
\end{equation}
Now suppose that $S \in \mathcal{V}_{u,v}^{0,1}$, for which we have $S^\prime = S$ and $B_{G_{e}}(S) = B_{G}(S^\prime)\cup \{u\}$. 
Because of the bijection between $\mathcal{V}_{u,v}^{0,1}$ and the collection $\{S^\prime \subseteq V(G): {\bf 1}_{S^\prime}(v) = 1\}$, then we have 
\begin{equation}\label{eq:BGE2}
    B(G_{e};x,y)_{u^{1}}^{0} = \sum_{S \in \mathcal{V}_{u,v}^{0,1}}x^{|B_{G_{e}}(S)|}\,y^{|S|} = \sum_{S^\prime \subseteq V(G),{\bf 1}_{S^\prime}(v) = 1}x^{|B_{G}(S^\prime)| + 1}\,y^{|S^\prime|} = xB(G;x,y)_{v}^{1}.
\end{equation}
Let $S \in \mathcal{V}_{u,v}^{1,1}$. Then $|S| = |S^\prime| + 1$ and $B_{G_{e}}(S) = B_{G}(S^\prime)$. 
From the bijection between $\mathcal{V}_{u,v}^{1,1}$ and the collection $\{S^\prime \subseteq V(G): {\bf 1}_{S^\prime}(v) = 1\}$, we obtain,
\begin{equation}\label{eq:BGE3}
    B(G_{e};x,y)_{u^{1}}^{1} = \sum_{S \in \mathcal{V}_{u,v}^{1,1}}x^{|B_{G_{e}}(S)|}\, y^{|S|} = \sum_{S^\prime \subseteq V(G),{\bf 1}_{S^\prime}(v) = 1}x^{|B_{G}(S^\prime)|}\,y^{|S^\prime|+1} = yB(G;x,y)_{v}^{1}.
\end{equation}
Let $S \in \mathcal{V}_{u,v}^{1,0}$. Hence, $|S| = |S^\prime|+1$. Thus, $B_{G_{e}}(S) = B_{G}(S^\prime)\cup\{u\}$ whenever $N_{G}(v)\cap S^\prime = \emptyset$ and $B_{G_{e}}(S) = B_{G}(S^\prime)$ whenever $N_{G}(v)\cap S^\prime \not= \emptyset$. Due to the bijection between $\mathcal{V}_{u,v}^{1,0}$ and $\{S^\prime \subseteq V(G):1_{S^\prime}(v) = 0,N_{G}(v)\cap S^\prime = \emptyset\} \cup \{S^\prime \subseteq V(G): {\bf 1}_{S^\prime}(v) = 0,N_{G}(v)\cap S^\prime \not= \emptyset\}$. Then,
\begin{equation}\label{eq:BFE4}
    B(G_{e};x,y)_{u^{0}}^{1}  = \sum_{S \in \mathcal{V}_{u,v}^{1,0}}x^{|B_{G_{e}}(S)|}\,y^{|S|} = xyB(G;x,y)_{v^{0}}^{0}+yB(G;x,y)_{v^{1}}^{0}
\end{equation}
    
\end{proof}

The following result shows an implicit formula for the boundary polynomial of path graphs.

\begin{proposition}\label{p:Pn}
 For every $n\ge 0$
 \begin{equation}\label{eq:Pn}
 B(P_n;x,y)=  \left(\begin{array}{cccc}1&1&1&1\end{array}\right)
 \left(\begin{array}{cccc} 1 & 1 & 0 & 0 \\ 0 & 0 & x & x \\ xy & y & 0 & 0 \\ 0 & 0 & y & y \end{array}\right)^{n}
 \left(\begin{array}{c}0\\1\\0\\0\end{array}\right).
 \end{equation}
\end{proposition}

\begin{proof}
Since $P_1 \simeq E_1$ we have $B(P_1;x,y)=B(E_1;x,y)=y+1$
and

 \begin{equation}
\nonumber
 \left(\begin{array}{cccc}1&1&1&1\end{array}\right)
 \left(\begin{array}{cccc} 1 & 1 & 0 & 0 \\ 0 & 0 & x & x \\ xy & y & 0 & 0 \\ 0 & 0 & y & y \end{array}\right)
 \left(\begin{array}{c}0\\1\\0\\0\end{array}\right)=1+y.
 \end{equation}

Let $V(P_n)=\{v_1\ldots,v_n\}$ with $v_{i}v_{i+1} \in E(P_n)$ for $1\le i\le n-1$.
By Lemma \ref{l:pendent+} we have
  \[
 \left(\begin{array}{c} B(P_n;x,y)^{0}_{v^{0}_n} \\ B(P_n;x,y)^{0}_{v^{1}_n} \\ B(P_n;x,y)^{1}_{v^{0}_n} \\ B(P_n;x,y)^{1}_{v^{1}_n} \end{array}\right)
 =
 \left(\begin{array}{cccc} 1 & 1 & 0 & 0 \\ 0 & 0 & x & x \\ xy & y & 0 & 0 \\ 0 & 0 & y & y \end{array}\right)
 \left(\begin{array}{c} B(P_{n-1};x,y)^{0}_{v^{0}_{n-1}} \\ B(P_{n-1};x,y)^{0}_{v^{1}_{n-1}} \\ B(P_{n-1};x,y)^{1}_{v^{0}_{n-1}} \\ B(P_{n-1};x,y)^{1}_{v^{1}_{n-1}} \end{array}\right).
 \]
 Applied Lemma \ref{l:pendent+} to the right part of the equality one more time, we get

 \[
 \left(\begin{array}{c} B(P_n;x,y)^{0}_{v^{0}_n} \\ B(P_n;x,y)^{0}_{v^{1}_n} \\ B(P_n;x,y)^{1}_{v^{0}_n} \\ B(P_n;x,y)^{1}_{v^{1}_n} \end{array}\right)
 =
 \left(\begin{array}{cccc} 1 & 1 & 0 & 0 \\ 0 & 0 & x & x \\ xy & y & 0 & 0 \\ 0 & 0 & y & y \end{array}\right)\left(\begin{array}{cccc} 1 & 1 & 0 & 0 \\ 0 & 0 & x & x \\ xy & y & 0 & 0 \\ 0 & 0 & y & y \end{array}\right) \left(\begin{array}{c} B(P_{2};x,y)^{0}_{v^{0}_{n-2}} \\ B(P_{2};x,y)^{0}_{v^{1}_{n-2}} \\ B(P_{2};x,y)^{1}_{v^{0}_{n-2}} \\ B(P_{2};x,y)^{1}_{v^{1}_{n-2}} \end{array}\right)
 \]

Thus, by applying Lemma \ref{l:pendent+} subsequently we obtain
  \[
 \left(\begin{array}{c} B(P_n;x,y)^{0}_{v^{0}_n} \\ B(P_n;x,y)^{0}_{v^{1}_n} \\ B(P_n;x,y)^{1}_{v^{0}_n}\\ B(P_n;x,y)^{1}_{v^{1} _n} \end{array}\right)
 =
 \left(\begin{array}{cccc} 1 & 1 & 0 & 0 \\ 0 & 0 & x & x \\ xy & y & 0 & 0 \\ 0 & 0 & y & y \end{array}\right)^{n}
 \left(\begin{array}{c} 0\\1\\0\\0\end{array}\right).
 \]

By (left) multiplication of
$\left(\begin{array}{cccc} 1 & 1 & 1 &1 \end{array}\right)$ 
we obtain the result for $n\ge2$. 
We may consider that \eqref{eq:Pn} holds for $n=0$,
too. Note that if $G=(\emptyset,\emptyset)$ we have
$B(G;x)=x^{n_\emptyset+\p(\emptyset)}=1=(1\ 1\ 1\ 1)\,(0\ 1\ 0\
0)^T$.
\end{proof}

Define as $G^{\prime}$ the graph obtained from $G$ by joining with a new edge two non-adjacent vertices of $G$. For $u,v \in V(G)$ and $i,j,k \in \{0,1\}$, we define the following, where $\widetilde{A}_{i,j} = \{S\subseteq V(G)\,:\,{\bf 1}_{S}(u)=i,{\bf 1}_{S}(v)=j,(S\setminus \{v\})\cap N(u) \not= \emptyset\}$ and $\widetilde{B}_{i,j} = \{S\subseteq V(G)\,:\,{\bf 1}_{S}(u)=i,{\bf 1}_{S}(v)=j,(S\setminus \{v\})\cap N(u) = \emptyset\}$, and $\widetilde{C}_{i,j} = \{S\subseteq V(G)\,:\,{\bf 1}_{S}(u)=i,{\bf 1}_{S}(v)=j\}$.

$$B(G;x,y)_{u^{1},v}^{i,j} := \sum_{S\subseteq\widetilde{A}_{i,j}}x^{|B(S)|} y^{|S|},$$

$$B(G;x,y)_{u^{0},v}^{i,j} := \sum_{S\subseteq\widetilde{B}_{i,j}}x^{|B(S)|}y^{|S|},$$

$$B(G;x,y)_{u,v^{k}}^{i,j} := B(G;x,y)_{v^{k},u}^{j,i},$$

$$B(G;x,y)_{u,v}^{i,j} := \sum_{S\subseteq\widetilde{C}_{i,j}}x^{|B(S)|} y^{|S|},$$

\begin{remark}\label{re:1}
    The following equalities follow from previous definitions,
    $$B(G;x,y) = B(G;x,y)_{u,v}^{0,0} + B(G;x,y)_{u,v}^{1,0} + B(G;x,y)_{u,v}^{0,1} + B(G;x,y)_{u,v}^{1,1},$$
    $$B(G;x,y)_{u,v}^{1,0} = B(G;x,y)_{u,v^{0}}^{1,0} + B(G;x,y)_{u,v^{1}}^{1,0},$$
    $$B(G;x,y)_{u,v}^{0,1} = B(G;x,y)_{u^{0},v}^{0,1}+B(G;x,y)_{u^{1},v}^{0,1}.$$
\end{remark}

Let $G$ be a graph, and let $e\in E$. Consider now the graph $G-e$ obtained by removing the edge $e$ from $G$. The following result gives an algebraic relation between the polynomials $B(G;x,y)$ and $B(G-e;x,y)$.

\begin{theorem}\label{th:G+e}
    Let $G$ be a graph with an edge $e$ joining the adjacent vertices $u,v\in V(G)$. Then, we have

    \begin{equation}\label{eq:G+e}
        B(G;x,y)-B(G-e;x,y)= (x-1) \left( B(G-e;x,y)^{0,1}_{u^{0},v} + B(G-e;x,y)^{1,0}_{u,v^{0}}
 \right).
    \end{equation}
\end{theorem}

\begin{proof}
    Consider $S \subseteq V(G)$. If $u,v \in S$ or $u,v \not\in S$, then $\partial_{G}(S) = \partial_{G-e}(S)$. Therefore is clear that,
    $$B(G;x,y)_{u,v}^{1,1} = B(G-e;x,y)_{u,v}^{1,1} \qquad\qquad \text{ and } \qquad\qquad  B(G;x,y)_{u,v}^{0,0} = B(G-e;x,y)_{u,v}^{0,0}.$$
    Now, assume that $\{u,v\}\cap S = \{u\}$, then if $(S\setminus\{u\})\cap N(v) = \emptyset$, then $B_{G}(S) = B_{G-e}(S)\cup\{v\}$ and $|B_{G}(S)| = |B_{G-e}(S)|+1$. If $(S\setminus \{u\})\cap N(v) \not= \emptyset$, then $B_{G}(S) = B_{G-e}(S)$. Then, from the previous definitions, we have
    $$B(G;x,y)_{u,v}^{1,0} = xB(G-e;x,y)_{u,v^{0}}^{1,0}+B(G-e;x,y)_{u,v^{1}}^{1,0}.$$
    Analogously, for $\{u,v\}\cap S = \{v\}$, we have
    $$B(G;x,y)_{u,v}^{0,1} = xB(G-e;x,y)_{u^{0},v}^{0,1}+B(G-e;x,y)_{u^{1},v}^{0,1}.$$
    Finally, Remark \ref{re:1} yields the desired result.
\end{proof}

The result above has the following direct consequences.

\begin{proposition}\label{p:Cn-Pn}
    Let $P_n,C_n$ be the path and cycle graphs with vertices $\{v_1,v_2,\ldots,v_n\}$ such that $v_i\sim v_{i+1}$ for every $1\leq i<n$. Then, we have 

    \begin{equation}\label{eq:Cn-Pn}
        B(C_n;x,y)=B(P_n;x,y) + 2(x-1) \, B(P_n;x,y)^{0,1}_{v_1^{0},v_n}.
    \end{equation}
\end{proposition}

\begin{proposition}\label{p:G1-G2}
Let $G_1$ and $G_2$ be two disjoint graphs. Consider the graph $G$ obtained by adding a new edge $uv$ to $G_1\uplus G_2$, where $u\in V(G_1)$ and $v\in V(G_2)$. Then, 

\begin{equation}
 B(G;x,y)
 =\displaystyle
 \left(\begin{array}{ccc} B(G_1;x,y)^{1}_{u} & B(G_1;x,y)^{0}_{u^{0}} & B(G_1;x,y)^{0}_{u^{1}} \end{array}\right)
 \left(\begin{array}{ccc} 1 & x & 1 \\ x & 1 & 1 \\ 1 & 1 & 1 \end{array}\right)
 \left(\begin{array}{c} B(G_2;x,y)^{1}_{v} \\ B(G_2;x,y)^{0}_{v^{0}} \\ B(G_2;x,y)^{0}_{v^{1}}\end{array}\right).
\end{equation}

\end{proposition}

\begin{proof}
Let $n_1$, $n_2$ and $n$ be the order of the graphs $G_1$, $G_2$ and $G$, respectively. Hence, $n=n_1+n_2$. For every $S \subseteq V(G)$ we define
$S_{1} := S \cap V(G_1)$ and $S_{2} := S \cap V(G_2)$, thus $S_1\cap S_2=\emptyset$. Note that if $u,v \in S$, then $B_{G}(S)= B_{G_{1}}(S_{1}) \cup B_{G_{2}}(S_{2})$, and so
$|S|= |S_{1}| + |S_{2}|$. Then
\begin{equation}
\nonumber
B(G;x,y)^{1,1}_{u,v} = B(G_{1};x,y)^{1}_{u}\,B(G_{2};x,y)^{1}_{v}.
\end{equation}

Similarly,  if $u,v \notin S$, we deduce
\begin{equation}
\nonumber
B(G;x,y)^{0,0}_{u,v} = B(G_{1};x,y)^{0}_{u}B(G_{2};x,y)^{0}_{v}.
\end{equation}

Now assume that $u \in S$ and $v \notin S$. If $(S\setminus\{u\})\cap N(v)=\emptyset$,  then
$B_{G}(S)= B_{G_{1}}(S_{1}) \cup B_{G_{2}}(S_{2}) \cup \{v\}$, and  $x\,B(G_{1};x,y)^{1}_{u}\,B(G_{2};x,y)^{0}_{v^{0}}$ is a summation term in $B(G;x,y)^{1,0}_{u,v}$.
If $(S\setminus\{u\})\cap N(v)\neq \emptyset$, then $B_{G}(S)= B_{G_{1}}(S_{1}) \cup B_{G_{2}}(S_{2})$, and 
$B(G_{1};x,y)^{1}_{u}\,B(G_{2};x,y)^{0}_{v^{1}}$ is the remaining term in $B(G;x,y)^{1,0}_{u,v}$. Therefore
\begin{equation}
\nonumber
B(G;x,y)^{1,0}_{u,v} = B(G_{1};x,y)^{1}_{u}\Big(x\,B(G_{2};x,y)^{0}_{v^{0}}
+ B(G_{2};x,y)^{0}_{v^{1}}\Big).
\end{equation}
Similarly, if  $v \in S$ and $u \notin S$, we deduce
\begin{equation}
\nonumber
B(G;x,y)^{0,1}_{u,v} = B(G_{2};x,y)^{1}_{v} \Big(x\,B(G_{1};x,y)^{0}_{u^{0}}
+ B(G_{1};x,y)^{0}_{u^{1}}\Big).
\end{equation}
The desired result follows by combining these four equations above.
\end{proof}

Let $V(G) = \{v_{1},\ldots,v_{n}\}$ and let $\mathcal H = \{H_{i}\}_{i=1}^{n}$ be a family of graphs. The corona product $G \odot \mathcal H$ is the graph obtained by joining the $j$-th vertex of $G$ with all the vertices of $H_{j}$. We have that $V(G \odot \mathcal H) = V(G)\cup\left(\cup_{j=1}^{n}V(H_{j})\right)$. We can clearly see that the corona product is a non-commutative and non-associative operation. In particular, we have that $E_{1} \odot \{G\}$ is isomorphic to the graph $E_{1} + G$. Our next result provides a relation between the boundary polynomials of $P_2\odot\{G_1,G_2\}$, $E_1+G_1$ and $E_1+G_2$.

\begin{theorem}\label{t:DobleJoin}
 Let $G_1,G_2$ be two graphs of orders $n_1,n_2\ge 1$, respectively. Then
  \begin{eqnarray}
  \nonumber
  B(P_2\odot\{G_1,G_2\}; x,y) &=& B(E_1+G_1; x,y)\cdot B(E_1+G_2; x,y)\\
  &&+y(x-1)\Big( (x+y)^{n_{1}}+(x+y)^{n_{2}} \Big).
  \end{eqnarray}
\end{theorem}


\begin{proof}
Let $H:=P_2\odot\{G_1,G_2\}, V(P_2)=\{v_1,v_2\}$ and $|V(H)|=n_1+n_2+2$. Then $H-v_1v_2\simeq (E_1+G_1)\uplus(E_1+G_2)$. By Theorems \ref{t:factors} and \ref{th:G+e} we have
\[
B(H; x,y) - B(E_1+G_1; x,y)\, B(E_1+G_2; x,y) = (x-1)\big( B(H-v_{1}v_{2};x,y)_{v_{1},v_{2}^{0}}^{1,0}+B(H-v_{1}v_{2};x,y)_{v_{1}^{0},v_{2}}^{0,1}\big).
\]
Let us now calculate $B(H-v_{1}v_{2};x,y)_{v_{1},v_{2}^{0}}^{1,0}$. Then, consider $S \subseteq V(H-v_{1}v_{2})$ with $v_{1} \in S$, $v_{2} \notin S$ and $S\cap N(v_{2}) = \emptyset$. Hence, we have $S \subseteq (V(G_{1})\cup\{v_{1}\})$, $|B(S)| = n_{1}+1-|S|$.  Since the number of such
subsets $S$ with size $i+1$ is ${n_1 \choose i}$, from (\ref{eq:Poly2}), we have
\begin{equation}
\nonumber
B(H-v_1v_2;x,y)_{v_1,v^{0}_2}^{1,0}= \sum_{i=0}^{n_{1}}\binom{n_{1}}{i}x^{n_{1}-i}y^{i+1} = y\,(x+y)^{n_{1}}.
\end{equation}

Analogously, we have
\begin{equation}
\nonumber
B(H-v_1v_2;x,y)_{v^{0}_1,v_2}^{0,1} = B(H-v_{1}v_{2};x,y)_{v_{2}^{0},v_{1}}^{1,0} = y\,(x+y)^{n_{2}}.
\end{equation}
\end{proof}

Consider two star graphs $S_{r},S_{t}$ with central vertices $v_r$ and $w_t$, respectively.
  The {\em double star graph} $S_{r,t}$ is the graph obtained by joining $v_r$ to $w_t$ with an additional edge.
 Then $S_{r,s}\simeq P_2\odot\{E_{r-1},E_{s-1}\}$ and Theorem \ref{t:DobleJoin} implies the following result.


\begin{proposition}\label{p:S_rt}
For $r,t\geq 2$, we have
\[
B(S_{r,t}; x,y) = B(S_r; x,y)\, B(S_t; x,y) + y(x-1)\Big( (x+y)^{r-1} + (x+y)^{t-1} \Big).
\]
\end{proposition}

Let us analyze the distortion of the boundary polynomial by the edge subdivision for a more complete analysis under local transformations of a graph.

\begin{theorem}
Let $e$ be an edge of $G$ incident with vertices $u$ and $v$. If $G'$ is the graph that results by subdividing  $e$ once, then
\begin{equation}
\begin{aligned}
\nonumber
B(G';x,y) & = B(G;x,y)^{0,0}_{u,v}+B(G;x,y)^{1,0}_{u,v^{0}}+x\,B(G;x,y)^{1,0}_{u,v^{1}}
+ B(G;x,y)^{0,1}_{u^{0},v}+x\,B(G;x,y)^{0,1}_{u^{1},v}\\
           &+x\,B(G;x,y)^{1,1}_{u,v}+x^{2}y\,B(G;x,y)^{0,0}_{u^{0},v^{0}}+xy\,B(G;x,y)^{0,0}_{u^{1},v^{0}}+
xy\,B(G;x,y)^{0,0}_{u^{0},v^{1}}\\
           & +y\,B(G;x,y)^{0,0}_{u^{1},v^{1}}+yB(G;x,y)^{1,1}_{u,v}+
y\,B(G;x,y)^{1,0}_{u,v}+y\,B(G;x,y)^{0,1}_{u,v}.
\end{aligned}
\end{equation}
\end{theorem}

\begin{proof} Let $w$ be the new subdividing vertex of $G'$.  Then $V(G')=V(G)\cup \{w\}, |V(G')|=n+1$, and
$N(w)=\{u,v\}$. Let
$S' \subseteq V(G')$, $S:= S' \setminus \{w\} \subseteq V(G)$. 

We first calculate $B(G';x,y)^{0}_{w^{0}}$.
If $u,v,w\notin S'$, then $B_{G'}(S')=B_{G}(S)$ and $|S|=|S'|$.
Then
\[
B(G';x,y)^{0}_{w^{0}}= B(G;x,y)^{0,0}_{u,v}.
\]

Similarly, if $w\in S'$ and $\{u,v\}\cap S'\neq \emptyset$, then $B_{G'}(S')=B_{G}(S)$ and $|S'|=|S|+1$.
Then
\[
B(G';x,y)^{1}_{w^{1}}= y \Big( B(G;x,y)^{1,1}_{u,v} + B(G;x,y)^{1,0}_{u,v} + B(G;x,y)^{0,1}_{u,v} \Big).
\]

We compute now $B(G';x,y)^{0}_{w^{1}}$. 
Note that if $u,v\in S'$ and $w \notin S'$, then  $B_{G'}(S')=B_{G}(S) \cup \{w\}$ and $|S|=|S'|$. 
Then
$x\,B(G;x,y)^{1,1}_{u,v}$ is an addend of $B(G';x,y)^{0}_{w^{1}}$.

Assume that $u \in S'$ and $v,w \notin S'$. If $N_{G'}(v)\cap S'=\emptyset$, then
$B_{G'}(S')=(B_{G}(S)\setminus \{v \}) \cup \{ w\}$, \emph{i.e.}, $|B_{G'}(S')|=|B_{G}(S')|$ and $|S|=|S'|$. Then $B(G;x,y)^{1,0}_{u,v^{0}}$
is an addend of $B(G';x,y)^{0}_{w^{1}}.$  
Now if $N_{G'}(v)\cap S' \neq \emptyset$, then
$B_{G'}(S')=B_{G}(S) \cup \{w \} $ and $|S|=|S'|$. Then $x\,B(G;x,y)^{1,0}_{u,v^{1}}$ is another addend of $B(G';x,y)^{0}_{w^{1}}$. Thus, by switching the roll of $u$ and $v$ we obtain
\[
 B(G';x,y)^{0}_{w^{1}} = x\,B(G;x,y)^{1,1}_{u,v} + B(G;x,y)^{1,0}_{u,v^{0}} + x\,B(G;x,y)^{1,0}_{u,v^{1}} + B(G;x,y)^{0,1}_{u^{0},v} + x\,B(G;x,y)^{0,1}_{u^{1},v}
\]

Finally, we compute $B(G';x,y)^{1}_{w^{0}}$. Hence, we have $u,v\notin S'$ and $w \in S'$. 
If $N(u)\cap S=\emptyset$ and $N(v)\cap S=\emptyset$ then  $B_{G'}(S')=B_{G}(S) \cup \{u,v\}$ and $|S'|=|S|+1$. 
Then
$x^2y\,B(G;x,y)^{0,0}_{u^0,v^0}$ is an addend of $B(G';x,y)^{1}_{w^{0}}$.
If $N(u)\cap S\neq\emptyset$ and $N(v)\cap S=\emptyset$ then  $B_{G'}(S')=B_{G}(S) \cup \{v\}$ and $|S'|=|S|+1$. 
Then
$xy\,B(G;x,y)^{0,0}_{u^1,v^0}$ is an addend of $B(G';x,y)^{1}_{w^{0}}$.
Thus, by switching the roll of $u$ and $v$ we also obtain that $xy\,B(G;x,y)^{0,0}_{u^0,v^1}$ is an addend of $B(G';x,y)^{1}_{w^{0}}$.
Hence, if $N(u)\cap S\neq\emptyset$ and $N(v)\cap S\neq\emptyset$ then  $B_{G'}(S')=B_{G}(S)$ and $|S'|=|S|+1$. 
Then
$y\,B(G;x,y)^{0,0}_{u^1,v^1}$ is an addend of $B(G';x,y)^{1}_{w^{0}}$.
Therefore, we have
\[
 B(G';x,y)^{1}_{w^{0}}= x^{2}y\,B(G;x,y)^{0,0}_{u^{0},v^{0}} +xy\,B(G;x,y)^{0,0}_{u^{1},v^{0}}+xy\,B(G;x,y)^{0,0}_{u^{0},v^{1}}
+ y\,B(G;x,y)^{0,0}_{u^{1},v^{1}}.
\]

\end{proof}

The following is an interesting (desired) result about the boundary polynomial of proper subgraphs.

\begin{theorem}\label{t:subgraph}
If $G'$ is a proper subgraph of $G$, then $B(G;x,y)\neq B(G';x,y)$.
\end{theorem}

\begin{proof}
If  $|V(G')|<|V(G)|$, Theorem \ref{t:eval} $i)$ implies $B(G;1,1)\neq B(G';1,1)$, and so
$B(G;x,y)\neq B(G';x,y)$. Thus we can assume $|V(G)|=|V(G')|$ and
$E(G')\subsetneq E(G)$. Then $G$ contains a nonempty edge set $\{e_1,\ldots,e_r\}$ such that
 $G'=G-\{e_1,\ldots,e_r\}$. The result
follows from applying $r$ times the Theorem \ref{th:G+e} (one time per additional edge in $G$) since $B(G;x,y)- B(G';x,y)= (x-1)\,P(x,y)$ where $P(x,y)$ is a non-null polynomial with positive coefficients.
\end{proof}

We say that a nonempty class of graphs $\mathcal K$ is {\em characterized} by a graph polynomial $f$ if
$f(G;x) = f(H;x)$ for any $G,H\in \mathcal K$.  
The following result is a consequence of Theorem \ref{t:eval}, Corollary \ref{c:E+G} and Theorem \ref{t:subgraph}.

\begin{theorem}\label{t:Kn+En+Kn-e}
  The classes of complete, complete without one edge, empty, path, cycle, wheel, star, double star graphs are characterized by their boundary polynomials.
\end{theorem}

\begin{proof}
    Note that Theorem \ref{t:eval} part iv states that the boundary polynomial of a graph encodes the degree sequence of the graph. Besides, we know that the complete, empty, star and double star graphs are characterized by its degree sequence; thus, we obtain the result for those families of graph.
    The path graph is not determined by its degree sequence, but it is characterized by its connectivity and degree sequence together; therefore, Theorem \ref{t:eval} part ii completes the result for the path graphs. 
    Consider now a graph $G$ such that $B(G;x,y)\simeq B(K_n-e;x,y)$. By Theorem \ref{t:eval} we have that $G$ has order $n$ and Theorem \ref{t:subgraph} give that $G\not\simeq K_n$ and $G$ is not isomorphic to any other graph with order $n$ other than $K_n-e$. 
    Finally, if a graph $G$ satisfies $B(G;x,y)=B(W_n;x,y)$, then $G$ has a vertex of degree $n-1$. In fact, $G\simeq E_1+H$ for some graph $H$ with order $n-1$. Then, by Corollary \ref{c:E+G} we have $B(H;x,y)=B(C_{n-1};x,y)$. Therefore, $H\simeq C_{n-1}$, and consequently, $G\simeq W_n$.
\end{proof}


\begin{figure}[h]
	\begin{center}
		\includegraphics[scale=0.5]{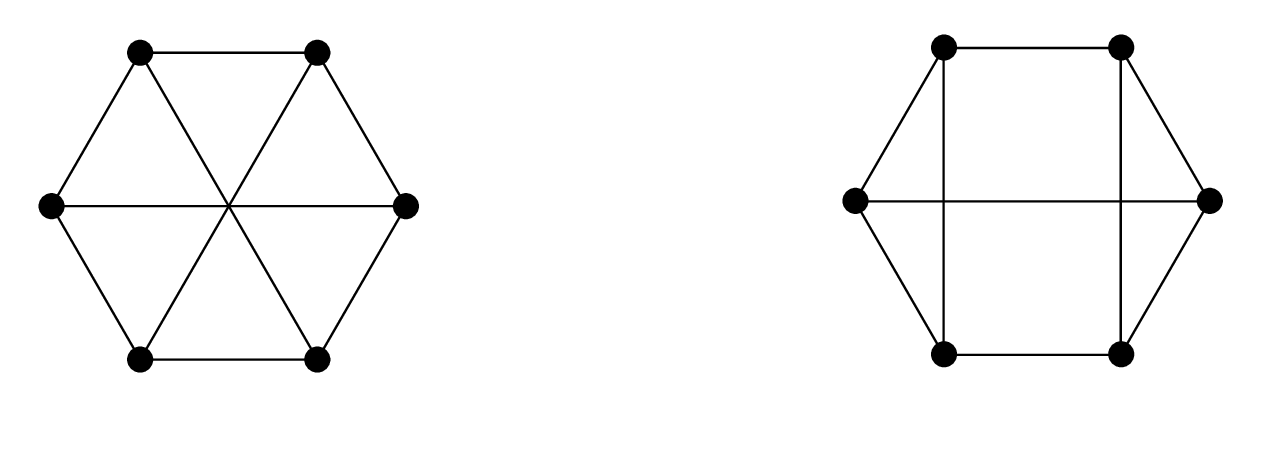}
	\end{center}
\caption{Non-isomorphic graphs with same order, size, degree sequence and boundary polynomial.} \label{fig:Cub6}
\end{figure}


The proof of the last result above gives that every graph that is determined by its degree sequence is also characterized by the boundary polynomial, and also the result is true when we add the hypothesis of being connected. Clearly, many other graphs are also characterized by the boundary polynomial; however, it is not true for all graphs. 
Figure \ref{fig:Cub6} shows two non-isomorphic graphs $G_1$ and $G_2$ with the same order, size, degree sequence, number of connected components, differential, domination number, Roman domination number, vertex connectivity, and also the same boundary polynomial. Indeed, these two graphs are distinguished by their alliance and characteristic polynomials among others.
A simple computation gives $B(K_{3,3};x,y)=B(P_2\Box C_3;x,y)=1 + 6 x^3y + 6 x^3y^2 + 9 x^4y^2 + 20 x^3y^3 + 15 x^2y^4 + 6xy^5 +y^6$. Note that this particular pair of graphs and Theorem \ref{t:join} give an infinite pairs of non-isomorphic graphs which have the same boundary polynomial, \emph{i.e.}, $G+K_{3,3}$ and $G+(P_2\Box C_3)$ for whatever graph $G$.

\end{document}